\documentclass[reqno]{amsart}

\usepackage{amsmath}
\usepackage{amsfonts}
\usepackage{amsthm}
\usepackage{amssymb}
\usepackage{graphicx}
\usepackage{graphics}
\usepackage{bm}
\usepackage{dsfont}
\usepackage{color}
\usepackage{enumerate}
\usepackage[font=footnotesize]{caption}
\usepackage[all]{xy}
\usepackage{xypic}
\allowdisplaybreaks

\numberwithin{equation}{section}

\newtheorem{thm}{Theorem}[section]
\newtheorem{cor}[thm]{Corollary}
\newtheorem{lem}[thm]{Lemma}
\newtheorem{prop}[thm]{Proposition}
\theoremstyle{definition}
\newtheorem{rem}[thm]{Remark}

\newtheorem{rmk}[thm]{Remark}

\newcommand{\N}{\mathds{N}}

\newcommand{\Z}{\mathds{Z}}

\newcommand{\R}{\mathds{R}}

\newcommand{\diff}{\mathrm{d}}

\begin{document}

\title[A note on the Morse homology in Banach manifolds]{A note on the Morse homology for a class of functionals\\ in Banach spaces involving the $p$-Laplacian}

\author[L. Asselle]{Luca Asselle}
\address{Ruhr-Universit\"at Bochum, Universit\"atsstra\ss e 150, 44801, Bochum, Germany}
\email{luca.asselle@math.uni-giessen.de}

\author[M. Starostka]{Maciej Starostka}
\address{Gda\'nks University of Technology, Gabriela Narutowicza 11/12, 80233 Gda\'nsk, Poland
\newline
\indent Institut f\"ur Mathematik, Naturwissenschaftliche Fakult\"at II, Martin-Luther-Universit\"at Halle-Wittenberg, \newline 
\indent 06099 Halle (Saale), Germany}
\email{maciejstarostka@pg.edu.pl}

\date{\today}
\subjclass[2000]{to be determined.}

\begin{abstract}
In this paper we show how to construct Morse homology for an explicit class of functionals involving the $p$-Laplacian. 
The natural domain of definition of such functionals is the Banach space $W^{1,2p}_0(\Omega)$, where $p>n/2$ and $\Omega \subset \R^n$ is a bounded domain with sufficiently smooth boundary. As $W^{1,2p}_0(\Omega)$ is not isomorphic to its dual space,
critical points of such functionals cannot be non-degenerate in the usual sense, and hence in the construction of Morse homology we only require that the second differential at each critical point be injective. 
Our result upgrades the results in \cite{Cingolani:2003,Cingolani:2007}, where critical groups for an analogous class of functionals are computed, and provides in this special case a positive answer to Smale's suggestion that injectivity of the second differential should be enough for Morse theory. 
\end{abstract}

\maketitle





\section{Introduction}

Let $n\geq 2$. For $\Omega \subset \R^n$ bounded domain with sufficiently regular boundary, and for $p>n/2$, we consider the functional 
\begin{equation}
f: X:= W^{1,2p}_0(\Omega) \to \R, \quad f(u) := \frac 1{2p}  \int_{\Omega} \big (1+ |\nabla u|^2)^p\, \diff x + \int_\Omega G(u)\, \diff x,
\label{eq:functional}
\end{equation}
where $G:\R\to \R$ is a function of class $C^2$ such that
\begin{equation}
|G(t)| \leq \beta |t|^\alpha + \delta,  \quad \forall t\in \R,
\label{eq:growthG}
\end{equation}
for some  $\alpha \in [0,2p)$ and $\beta,\delta \geq 0$. 

Condition \eqref{eq:growthG} is needed to ensure that the functional  $f$ in \eqref{eq:functional} satisfies the Palais-Smale condition, a crucial property to do global critical point theory in infinite dimension. If one is interested in the computation of critical groups only, then such a condition can be removed, see \cite{Cingolani:2003} and references therein, where an analogous class of functionals is considered (we refer to the discussion after the statement of the main theorem for more details). 
Also, in this paper we are not interested in finding sharp conditions under which Morse homology can be defined. In this sense, the growth Condition \eqref{eq:growthG} as well as the condition $p>n/2$ can surely be relaxed. In the latter case however, one would have to face additional transversality issues, the reason being the fact that $W^{1,2p}_0(\Omega)$ admits $C^2$-smooth (actually smooth for $p=1$) bump functions only if $p\geq 1$. 

In this paper we show that Morse homology for functionals as in \eqref{eq:functional} is well-defined provided that all critical points $\bar u$ of $f$ are non-degenerate in the sense that the second differential of $f$ at $\bar u$ defines an injective linear operator $\mathrm{d}^2 f(\bar u):X \to X^*$. We shall stress the fact that such a condition is in general not enough to construct Morse homology (actually, not even to compute critical groups), as it does not even imply that the critical points are isolated, see e.g. \cite{Tromba:1977}.

\begin{thm}
Let $f:X\to \R$ be a functional as in \eqref{eq:functional} such that all critical points of $f$ are non-degenerate in the sense that the second differential $\mathrm d^2 f(\bar u): X\to X^*$ is 
injective for all $\bar u \in \mathrm{crit}\, (f)$. Then, Morse homology with $\Z_2$-coefficients for $f$ is well-defined and isomorphic to the singular homology of $X$, i.e. 
$$HM_* (f;\Z_2) \cong H_*(X;\Z_2) \cong \left \{\begin{array}{r} \Z_2 \quad \quad *=0, \\ 0 \ \quad \quad *\geq 1.\end{array}\right.$$
In particular, it is independent of $p>n/2$. 
\label{thm:main}
\end{thm} 

To our best knowledge, the theorem above represents the first concrete instance in which Morse homology in a Banach space setting is defined. For an abstract construction we refer to \cite{Abbondandolo:2006lk}. Functionals  as in \eqref{eq:functional} are interesting for at least two reasons: first, they are intimately 
related with the class of functionals considered in \cite{Cingolani:2003,Cingolani:2007}, whose critical points correspond to weak solutions of a quasi-linear problem (involving the $p$-Laplacian) 
which arises in the mathematical description of propagation phenomena of solitary waves. In fact, for this latter class of functionals the construction of Morse homology carries over word by word. 
Second, they are similar to the class of functionals of $\alpha$-harmonic maps, $\alpha>1$,  introduced in \cite{Sacks:1981} to prove the existence of harmonic maps by studying the convergence of $\alpha$-harmonic maps as $\alpha \downarrow 1$. In fact, our approach should allow to define Morse homology for such class of functionals as well.  

\vspace{2mm}

The strategy of the proof of Theorem \ref{thm:main} is the following: arguing as in \cite{Cingolani:2003,Cingolani:2007} one sees that the critical points of $f$ are isolated and have 
finite Morse index. Also, the growth Condition \eqref{eq:growthG} implies that $f$ satisfies the Palais-Smale condition on $X$. Therefore, in order to apply the abstract theory developed
in \cite{Abbondandolo:2006lk} one has to prove the existence of a $C^2$-smooth complete Morse (i.e. with only hyperbolic rest points) vector field $F$ on $X$ such that $f$ is a 
Lyapounov function for the flow induced by $F$,
see Section 2. The key ingredient here is a sort of uniform convexity of $f$ in the positive direction determined by the second differential $\mathrm{d}^2f $ at each critical point, see Lemma \ref{lem:uniformconvexity}. It would be interesting to check such a condition in other concrete examples. 

The fact that $f$ is a Lyapounov function for $F$ together with the fact that all critical points have finite Morse index (such a condition can be relaxed, see \cite[Theorem 1.20]{Abbondandolo:2006lk}) implies that the stable resp. unstable manifold $W^s(\bar u,F)$ resp. $W^u(\bar u,F)$ of a rest point $\bar u$ of $F$ (equivalently, of a critical point $\bar u$ of $f$) is an embedded $C^2$-submanifold homeomorphic to an open disc.  Since critical points of $f$ have finite Morse index and the pair $(f,F)$ 
satisfies the Palais-Smale condition, after a generic perturbation of 
$F$ we can achieve transverse intersection between stable and unstable manifolds of pairs of critical points whose Morse indices differ at most by two. Such intersections are therefore
finite dimensional pre-compact embedded submanifolds of $X$ of dimension equal the difference of the Morse indices. Now one argues as usually to define a Morse complex which is generated by critical points of $f$ and whose boundary operator counts the number of gradient flow lines (modulo two) between pairs of critical points whose Morse indices differ by one. For more details about the abstract construction of the Morse complex we refer to \cite[Section 2]{Abbondandolo:2006lk}. 

\begin{cor}
Let $f$ be as in Theorem \ref{thm:main}. Then, $f$ has either one or at least three critical points. 
\end{cor}

\begin{proof}
This is an immediate consequence of Theorem \ref{thm:main}. The functional $f$ has at least one critical point because it is bounded from below and satisfies the Palais-Smale condition. Also, the presence of a second critical point implies the existence of a third one in order for them to cancel 
out in homology. 
\end{proof}



\section{Construction of Morse homology}

The differential of $f$ as in \eqref{eq:functional} at $u$ is given by 
\begin{equation}
\diff f(u)[v] = \int_\Omega \big (1+|\nabla u|^2\big )^{p-1} \langle \nabla u,\nabla v\rangle \, \diff x + \int_\Omega G'(u) v\, \diff x.
\label{eq:differentialf}
\end{equation}
If $\bar u \in X$ is a critical point of $f$, then the second differential of $f$ at $\bar u$ is 
\begin{align*}
\diff^2 & f(\bar u) [v,w] \\
& = \int_\Omega \big (1+|\nabla \bar u|^2\big )^{p-1} \langle \nabla v,\nabla w\rangle \, \diff x + 2(p-1) \int_\Omega \big (1+|\nabla \bar u|^2)^{p-2} \langle \nabla \bar u,\nabla v\rangle \langle \nabla \bar u,\nabla w\rangle \, \diff x+ \int_\Omega G''(\bar u) vw\, \diff x.
\end{align*}
Hereafter we assume that $\bar u$ is a \textit{non-degenerate} critical point, in the sense that $\diff^2 f (\bar u):X\to X^*$ is injective. As easy examples show, such a condition is in general not sufficient 
to do Morse theory for abstract functionals on Banach manifolds, as it does not even imply that the critical point is isolated, see e.g. \cite{Tromba:1977}. However, for the class of functionals in \eqref{eq:functional}, critical groups can be defined in a similar way to \cite{Cingolani:2003} under such an assumption. In this paper, we upgrade such a result showing that full Morse homology can be defined. 

\begin{rem}
Critical points of  functionals as in \eqref{eq:functional} cannot be non-degenerate in the classical sense (i.e. $\diff^2 f(\bar u)$ isomorphism), since $X$ is not isomorphic to its dual space $X^*$. On the other hand, injectivity of the second differential at an isolated critical point can be obtained by arbitrarily small finite dimensional Marino-Prodi type perturbations, as shown in \cite[Theorem 1.6]{Cingolani:2007}. 
\end{rem}

In the next theorem we construct a $C^2$-smooth complete Morse vector field on $X$ for which $f$ is a Lyapounov function. This is the crucial step in the definition of Morse homology. 

\begin{thm}
Let $f$ be a functional as in \eqref{eq:functional} having only non-degenerate critical points in the sense above. Then, there exists a $C^2$-smooth vector field $F$ on $X$ such that:
\begin{enumerate}
\item[i)] $F$ is complete,
\item[ii)] $f$ is a Lyapounov function for $F$,
\item[iii)] $F$ is Morse, i.e. the Jacobian of $F$ at every critical point $\bar u$ of $f$ is an hyperbolic operator on $T_{\bar u} X$,
\item[iv)] $(f,F)$ satisfies the Palais-Smale condition,
\item[v)] $F$ satisfies the Morse-Smale condition up to order two. 
\label{thm:pseudogradient}
\end{enumerate}
\end{thm}

\begin{proof}[Proof of Theorem \ref{thm:main}]
We define a chain complex $(C_*(f),\partial )$ by setting 
$$C_k(f) := \bigoplus_{\mu(\bar u)=k} \Z_2 \langle \bar u \rangle,$$
where $\mu(\bar u)$ denotes the Morse index of $\bar u\in \text{crit}\, (f)$, and
$$\partial \bar u := \sum_{\mu (\bar v) = \mu (\bar u) -1} \Big (\big |\mathcal M(\bar u, \bar v)\big |   \ \text{mod} \ 2 \Big ) \cdot \bar v,$$
where $\mathcal M(\bar u,\bar v)$ is the intersection between the unstable manifold $W^u(\bar u, F)$ of $\bar u$ and the stable manifold $W^s(\bar v,F)$ of $\bar v$. 
Conditions i)-iv) in Theorem \ref{thm:main} imply that $W^u(\bar u, F)$ resp. $W^s(\bar v,F)$ is a finite dimensional resp. codimensional embedded $C^2$-submanifold of $X$ 
homeomorphic to a disk of dimension $\mu(\bar u)$ resp. of codimension $\mu(\bar v)$. The Morse-Smale condition up to order one then implies that $\mathcal M(\bar u,\bar v)$ is 
a pre-compact one-dimensional embedded submanifold, and as such consists of only finitely many $F$-flow lines connecting $\bar u$ and $\bar v$. In particular, $\partial \bar u$ is 
well-defined as there can be only finitely many critical points of $f$ of index $\mu(\bar u)-1$ contained in $f^{-1}(-\infty, f(\bar u))$. Finally, the Morse-Smale condition up to order two implies 
that $\partial^2 =0$, so that $(C_*(f),\partial)$ is indeed a chain complex. The fact that the resulting Morse homology is isomorphic to the singular homology of $X$ is proved in \cite[Theorem 2.8]{Abbondandolo:2006lk}.
\end{proof}

To prove Theorem \ref{thm:pseudogradient}, the first step is to relate the notion of non-degeneracy above with a notion of non-degeneracy which is more convenient for Morse homology, namely the 
existence of a linear hyperbolic operator $L$ on $X$ such that $f$ is a Lyapounov function for the linear flow defined by $L$ in a neighborhood of $\bar u$, see e.g. \cite{Abbondandolo:2006lk,Tromba:1977,Uhlenbeck:1972}. 

\begin{prop}
Let $\bar u \in X$ be a non-degenerate critical point of $f$ as in \eqref{eq:functional}. Then, there exist a neighborhood $\mathcal U$ of $\bar u$ in $X$ 
and a linear hyperbolic operator $L:T_{\bar u} X \to T_{\bar u} X$ such that, on $\mathcal U$, $f$ is a Lyapounov function for the linear flow defined by $L$.  
\label{prop:hyperbolic}
\end{prop}

To prove Proposition \ref{prop:hyperbolic} we first recall some facts which are proved in \cite{Cingolani:2003} for a slightly different class of functionals, but all proofs go through with minor modifications to the setting of the present paper. 
 Because of the embedding $X \hookrightarrow L^\infty(\Omega)$, the critical point $\bar u$ is obviously contained in $L^\infty(\Omega)$. The results in \cite{Tolksdorf:1983,Tolksdorf:1984} then imply that $\bar u \in C^1(\overline \Omega)$. 
Following \cite{Cingolani:2003}, on $C^\infty_0(\Omega)$ we introduce the scalar product 
$$\langle v,w\rangle_{\bar u} .:= \int_\Omega \big (1+|\nabla \bar u|^2\big )^{p-1} \langle \nabla v,\nabla w\rangle \, \diff x + 2(p-1) \int_\Omega \big (1+|\nabla \bar u|^2)^{p-2} \langle \nabla \bar u,\nabla v\rangle \langle \nabla \bar u,\nabla w\rangle \, \diff x,$$
and define the Hilbert space 
$$\mathbb H_{\bar u} := \overline{C^\infty_0(\Omega)}^{\langle \cdot ,\cdot \rangle_{\bar u}}.$$
It is easy to see that $\mathbb H_{\bar u}$ is isomorphic to $W^{1,2}_0(\Omega)$, and thus we have a continuous embedding $X\hookrightarrow \mathbb H_{\bar u}.$
Moreover, $\diff^2 f(\bar u): X\to X^*$ extends to an invertible operator $H_{\bar u}:\mathbb H_{\bar u} \to \mathbb H_{\bar u}$ (where we have identified $\mathbb H_{\bar u}^*$ with $\mathbb H_{\bar u}$ using Riesz' representation theorem). Indeed, 
$$H_{\bar u} = \text{id} + K,$$
where $K: \mathbb H_{\bar u} \to \mathbb H_{\bar u}, v \mapsto Kv$, is the compact operator uniquely defined by
$$\langle Kv,w\rangle_{\bar u} = \int_\Omega G''(\bar u) vw\, \diff x, \quad \forall w \in \mathbb H_{\bar u}.$$
Being $H_{\bar u}$ a compact perturbation of the identity, it has Fredholm index zero. Furthermore, $H_{\bar u}$ is self-adjoint and as such its spectrum is real and consists of the eigenvalue 1 (which has infinite multiplicity) and of eigenvalues different from 1 (with finite multiplicity) which accumulate to $1$. Accordingly, we have a natural $\langle \cdot,\cdot\rangle_{\bar u}$-orthogonal decomposition
$$\mathbb H_{\bar u} = \mathbb H^- \oplus \mathbb H^0 \oplus \mathbb H^+,$$ 
where $\mathbb H^0:= \ker H_{\bar u}$ and 
$$\mathbb H^- := \!\!\!\!\! \bigoplus_{\footnotesize{\begin{array}{c} \lambda \in \sigma(H_{\bar u}) \\ \lambda < 0\end{array}}} \!\!\!\!\! \ker \big (\lambda \cdot \text{id} - H_{\bar u}\big ),\qquad \mathbb H^+ :=  \!\!\!\!\! \bigoplus_{\footnotesize{\begin{array}{c} \lambda \in \sigma(H_{\bar u}) \\ \lambda > 0\end{array}}}  \!\!\!\!\! \ker \big (\lambda \cdot \text{id} - H_{\bar u}\big ),$$
are the negative resp. positive eigenspace of $H_{\bar u}.$ Clearly, the set of positive eigenvalues of $H_{\bar u}$ is uniformly bounded away from zero, thus we can find a constant $\mu >0$ such that 
\begin{equation}\langle H_{\bar u} v,v\rangle_{\bar u} \geq \mu \|v\|_{\bar u}^2, \quad \forall v \in \mathbb H^+.\label{eq:H+}\end{equation}
Also, $\dim \mathbb H^- \oplus \mathbb H^0 < +\infty$, and standard regularity theory implies that 
$$\mathbb H^- \oplus \mathbb H^0 \subset X,$$
see \cite[Theorems 8.15, 8.24, 8.29]{Gilbarg}. Consequently, we obtain a splitting 
$$X= \mathbb H^- \oplus \mathbb H^0 \oplus W,$$
where $W:= \mathbb H^+ \cap X$, and \eqref{eq:H+} implies that 
\begin{equation}
\diff^2 f(\bar u)[v,v] \geq \mu \|v\|_{\bar u}^2, \quad \forall v \in W.
\label{eq:W}
\end{equation}
Since by assumption $\diff^2 f(\bar u):X\to X^*$ is injective, we finally deduce that $L_{\bar u}$ is injective and thus an isomorphism, being Fredholm of index zero. In particular,
$$\mathbb H_{\bar u} = \mathbb H^- \oplus \mathbb H^+, \quad \text{and}\ \ X = \mathbb H^-  \oplus W.$$
In the next result we prove that \eqref{eq:W} holds for any $u$ in a sufficiently small neighborhood of $\bar u$, thus showing that $f$ is locally uniformly convex around $\bar u$ in the $W$-direction with respect to the $\|\cdot\|_{\bar u}$-norm. The proof, which is analogous to the one of \cite[Lemma 4.2]{Cingolani:2003}, is included for the reader's convenience.  

\begin{lem}
There exists $r>0$ and $\mu>0$ such that for any $u \in X$ with $\|u -\bar u\|<r$ we have 
\begin{equation}
\diff^2 f(u)[v,v] \geq \mu \|v\|_{\bar u}^2, \quad \forall v \in W.
\label{eq:d2fW}
\end{equation}
In particular, $\bar u$ is a strict minimum point of $f$ in the $W$-direction.
\label{lem:uniformconvexity}
\end{lem}
\begin{proof}
Assume that there exist sequences $(u_n)\subset X$ and $(v_n)\subset W$ such that $u_n \to \bar u$ in $X$, $\|v_n\|_{\bar u}=1$ for all $n\in \N$, and 
\begin{equation}
\liminf_{n\to +\infty}\  \diff^2 f(u_n)[v_n,v_n] \leq 0.
\label{eq:liminf}
\end{equation}
Up to taking a subsequence, $v_n$ weakly converges (thus also strongly in $L^2$) to $v_\infty \in \mathbb H^+$. Noticing that 
\begin{align*}
\diff^2 &f(u_n)[v_n,v_n] \\
	&= \int_\Omega \big (1+|\nabla u_n|^2\big )^{p-1} | \nabla v_n|^2 \diff x + 2(p-1) \int_\Omega \big (1+|\nabla u_n|^2)^{p-2} \langle \nabla u_n,\nabla v_n\rangle^2 \, \diff x+ \int_\Omega G''(u_n) v_n^2\, \diff x\\
	&\geq \| \nabla v_n\|_2^2 +  \int_\Omega G''(u_n) v_n^2\, \diff x\\
	&\geq c \|v_n\|_{\bar u}^2 +  \int_\Omega G''(u_n) v_n^2\, \diff x\\
	&= 1 +  \int_\Omega G''(u_n) v_n^2\, \diff x,
	\end{align*}
	we infer that $v_\infty\neq 0$, as otherwise this would contradict \eqref{eq:liminf}. We now set
	$$h(x,u,v) := \big (1+|\nabla u|^2\big )^{p-1} | \nabla v|^2 + 2(p-1) \big (1+|\nabla u|^2)^{p-2} \langle \nabla u,\nabla v\rangle^2,$$
	so that 
	$$\diff^2 f(u_n)[v_n,v_n] = \int_\Omega h(x,u_n,v_n) \, \diff x + \int_\Omega G''(u_n) v_n^2\, \diff x.$$
	Obviously, $h$ is non-negative, continuous, and $v\mapsto h(x,u,v)$ is convex for every $(x,u)$. Therefore, the result in \cite{Ioffe:1977} implies that 
	$$(u,v) \mapsto \int_\Omega h(x,u,v)\, \diff x$$
	is lower-semicontinuous with respect to the strong convergence in the $u$-direction and the weak convergence in the $v$-direction.
	Now, using Assumption \eqref{eq:liminf}, Equation \eqref{eq:H+}, and the fact that $v_n\to v_\infty$ in $L^2$, we conclude
	\begin{align*}
	0 &\geq \liminf_{n\to +\infty}\  \diff^2 f(u_n)[v_n,v_n] \\
	   &=  \liminf_{n\to +\infty} \left ( \int_\Omega h(x,u_n,v_n) \, \diff x + \int_\Omega G''(u_n) v_n^2\, \diff x\right )\\
	   &\geq \int_\Omega h(x,\bar u,v_\infty)\, \diff x + \int_\Omega G''(\bar u) v_\infty^2\, \diff x\\
	   &= \langle L_{\bar u} v_\infty,v_\infty\rangle_{\bar u} \\
	   &\geq \mu \|v_\infty\|_{\bar u}^2,
	   \end{align*}
	   clearly a contradiction, as $v_\infty\neq 0$. This shows that 
	   $$\liminf_{n\to +\infty}\  \diff^2 f(u_n)[v_n,v_n] >0$$
	   for all sequences $(u_n)\subset X$ such that $u_n\to \bar u$, and all sequences $(v_n)\subset W$ such that $\|v_n\|_{\bar u}=1$ for all $n\in \N$. 
	   We claim now that for all such sequences there exists $\mu>0$ such that 
	   $$\liminf_{n\to +\infty}\  \diff^2 f(u_n)[v_n,v_n] \geq \mu.$$
	   Clearly, this finishes the proof of the lemma. Assuming this is not the case, for every $m\in \N$ we find sequences $(u_n^{(m)})$ and $(v_n^{(m)})$ as above such that 
	   $$\liminf_{n\to +\infty}\  \diff^2 f(u_n^{(m)})[v_n^{(m)},v_n^{(m)}] < \frac 1m.$$
	   Therefore, for every $m\in \N$ we can find $n(m)\in \N$ such that 
	   $$\diff^2 f(u_{n(m)}^{(m)})[v_{n(m)}^{(m)},v_{n(m)}^{(m)}] < \frac 2m.$$
	   The sequences $(u_{n(m)}^{(m)})_{m\in \N}$ and $(v_{n(m)}^{(m)})_{m\in \N}$ also satisfy 
	   $$u_{n(m)}^{(m)} \to \bar u, \quad \|v_{n(m)}^{(m)}\|_{\bar u}=1, \ \forall m\in \N,$$
	   and by construction 
	   $$\lim_{m\to +\infty} \diff^2 f(u_{n(m)}^{(m)})[v_{n(m)}^{(m)},v_{n(m)}^{(m)}]  =0,$$
	   a contradiction. 
\end{proof}

\begin{rmk}
A similar statement as in Lemma \ref{lem:uniformconvexity} holds also in the $\mathbb H^-$-direction, namely there exist $r>0$ and $\mu>0$ such that for any $u\in X$ with $\|u-\bar u\|<r$ we have 
\begin{equation}
\diff^2 f(u) [v,v] \leq -\mu \|v\|_{\bar u}^2, \quad \forall v \in \mathbb H^-.
\label{eq:d2fH-}
\end{equation}
However, the proof in this case is elementary since $f$ is of class $C^2$ and $\mathbb H^-$ is finite dimensional, so that $\|\cdot\|_{\bar u}$ and $\|\cdot\|$ are equivalent on $\mathbb H^-$. The details are left to the reader. 
\end{rmk}

\begin{proof}[Proof of Proposition \ref{prop:hyperbolic}]
On 
$$T_{\bar u} X \cong X = \mathbb H^- \oplus W$$
we define the linear operator $L=(\text{id}, -\text{id})$, that is 
$$L x := L(x_- + x_W) := x_- - x_W, \quad \forall x = x_- + x_W \in X.$$
The operator $L$ is clearly hyperbolic.  We claim
that there exists a sufficiently small neighborhood $\mathcal U$ of $\bar u$ such that $f$ is a Lyapounov function for the linear flow defined by $L$,
meaning that, for every $x\neq 0 \in \mathcal U - \{\bar u\}$, 
$$t \mapsto f(\bar u + e^{tL} x)$$
is strictly monotone decreasing, or, equivalenty, that 
$$\mathrm{d} f(\bar u + x) [Lx] < 0, \quad \forall x \in \mathcal U\setminus \{0\}.$$ 
We have 
$$\mathrm d f(\bar u + x)[\cdot] =\int_0^1 \frac{\mathrm d}{\mathrm d s} \big (f(\bar u + sx)\big )[\cdot]\, \diff s = \int_0^1 \mathrm d^2 f(\bar u + sx)[\cdot,x]\, \mathrm d s,$$
choose $r,\mu>0$ such that \eqref{eq:d2fW} and \eqref{eq:d2fH-} hold, and compute for $x\in X$ with $\|x\|<r$: 
\begin{align*}
\mathrm d &f(\bar u + x)[Lx] \\ 
   &= \int_0^1 \mathrm d^2 f(\bar u + sx)[Lx,x]\, \mathrm d s\\ 
			    &= \int_0^1 \mathrm d^2 f(\bar u + sx)[x_- - x_W, x_- + x_W] \, \mathrm d s\\
			    &= \int_0^1\Big ( \mathrm d^2 f(\bar u +sx)[x_-, x_-] -  \mathrm d^2 f(\bar u +sx)[x_W,x_W ] + \underbrace{ \mathrm d^2 f(\bar u +sx)[x_-, x_W]-  \mathrm d^2 f(\bar u +sx)[x_W, x_-]}_{=0}\Big ) \mathrm d s\\
			    & \leq - \mu \|x_-\|_{\bar u}^2 - \mu \|x_W\|_{\bar u}^2\\ 
			    & \leq  - \mu \|x\|_{\bar u}^2,
\end{align*}
which implies the claim.
\end{proof}

We recall that $f:X\to \R$ of class $C^1$ is said to satisfy the \textit{Palais-Smale condition}, if any sequence $(u_n)\subset X$ such that $f(u_n)\to c$, for some $c\in \R$, and 
$\mathrm{d} f(u_n)\to 0$ admits a converging subsequence. By the continuity of the differential, any limit point of a Palais-Smale sequence is a critical point of $f$. 

\begin{rmk}The Palais-Smale  condition plays the role of compactness of sublevel sets, and as such is a crucial ingredient in infinite dimensional critical point theory. 
We shall however stress the fact that, when the Morse
index and co-index of critical points is infinite, that is when $f$ is \textit{strongly indefinite}, the Palais-Smale condition alone is not enough to construct Morse homology, 
as the intersection between stable and unstable manifolds of 
pair of critical points might not be pre-compact, not even if finite dimensional. In such cases, stronger conditions are needed, see e.g. \cite{AM:05,Asselle:2020b,Asselle:2022}, where 
Morse homology for an abstract class of strongly indefinite functionals on a Hilbert manifold resp. for the Hamiltonian action in cotangent bundles is defined. Also, classical Morse theory is in such cases of no help,
since the topology of sublevel sets does not change when crossing a critical point with infinite Morse index. 
Anyhow, this will not be the case in the present paper, since critical points of a functional as in \eqref{eq:functional} always have finite Morse index. 
\end{rmk}

We show now that the growth condition \eqref{eq:growthG} on the function $G$ implies that the functional $f$ in \eqref{eq:functional} satisfies the Palais-Smale condition on $X$. 

\begin{lem}
Let $f:X\to \R$ be a functional as in \eqref{eq:functional}. Then $f$ satisfies the Palais-Smale condition. 
\label{lem:ps}
\end{lem}

\begin{proof}
Let $(u_n)\subset X$ be a Palais-Smale sequence for $f$. 

\vspace{2mm}

\textit{Claim 1.} $(u_n)$ is bounded in $X$. 

\vspace{2mm}

\noindent Suppose by contradiction that $\|u_n\|\to +\infty$. By the very definition of $f$, Equation \eqref{eq:growthG}, Poincar\'e inequality, and H\"older inequality we have for some constant $\gamma>0$: 
\begin{align*}
|f(u_n)| &= \left | \int_\Omega (1+|\nabla u_n|^2)^{p}\, \mathrm{d} x + \int_\Omega G(u_n)\, \mathrm{d} x\right | \\
	& \geq \|\nabla u_n\|_{2p}^{2p} - \int_\Omega |G(u_n)| \, \mathrm{d} x\\
	& \geq \gamma \|u_n\|^{2p} - \beta \|u_n\|_{L^\alpha}^\alpha - \delta \mu(\Omega)\\
	&\geq \gamma \|u_n\|^{2p} - \beta \mu(\Omega)^{\alpha (2p-\alpha)/2p} \|u_n\|_{L^{2p}}^{\alpha^2/2p} - \delta \mu(\Omega)\\
	&\geq \gamma \|u_n\|^{2p} - \beta \mu(\Omega)^{\alpha (2p-\alpha)/2p} \|u_n\|^{\alpha^2/2p} - \delta \mu (\Omega) \\
	&\to +\infty,	\end{align*}
as by assumption $\alpha<2p$. This is clearly a contradiction, since $f(u_n)\to c$ for some $c\in \R$.

\vspace{2mm}

\textit{Claim 2.} $(u_n)$ admits a converging subsequence. 

\vspace{2mm}

\noindent  Since $(u_n)$ is bounded, up to a subsequence we can assume $u_n\rightharpoonup u$ for some $u\in X$, hence in particular $u_n\to u$ in $L^\infty(\Omega)$. 
In view of \eqref{eq:differentialf} we can write $\mathrm {d} f:X\to X^*$ as 
$$\mathrm{d} f = D + K,$$
where 
\begin{align*}
& D:X\to X^*, \quad D(u)[\cdot] := \int_\Omega (1+|\nabla u|^2)^{p-1} \langle \nabla u, \nabla \cdot \rangle \, \mathrm{d}x, \\
& K:X\to X^*, \quad K(u)[\cdot] := \int_\Omega G'(u) v \, \mathrm{d} x.
\end{align*}
As shown in \cite[Appendix B]{Benci}, the non-linear operator $D$ is invertible with continuous inverse $D^{-1}$. We claim that $K(u_n)\to K(u)$ in operator norm. This follows from the fact that 
$K$ is sequentially compact, meaning that for any weakly converging sequence $w_n\rightharpoonup w$ in $X$ the operators $K(w_n)$ converge in operator norm to $K(w)$. To see this we compute using H\"older inequality with conjugated exponents $q,2p$: 
\begin{align*}
\|K(w_n)-K(w)\| &= \sup_{\|v\|=1} |K(w_n)[v]-K(w)[v]| \\
			& = \sup_{\|v\|=1} \left | \int_\Omega \big (G'(w_n) - G'(w)\big ) v \, \mathrm{d} x\right |\\
			& \leq \sup_{\|v\|=1} \left (\int_\Omega \big |G'(w_n)-G'(w)\big |^{q}\, \mathrm{d} x \right )^{1/q} \|v\|_{L^{2p}}\\
			& \leq \left (\int_\Omega \big |G'(w_n)-G'(w)\big |^{q}\, \mathrm{d} x \right )^{1/q}\\
			&\leq  \mu(\Omega)^{1/q} \sup_{x \in \Omega} |G'(w_n(x))-G'(w(x))| \\
			&\to 0,
\end{align*}
as $w_n\to w$ in $L^\infty(\Omega)$ and $G'$ is continuous. Now, since $(u_n)$ is a Palais-Smale sequence we have 
$$o(1) = \| \mathrm{d} f(u_n)\| = \|D(u_n) + K(u_n)\|,$$
thus $D(u_n) \to - K(u)$ in operator norm, and finally $u_n\to - D^{-1} (K(u))$ as $D^{-1}$ is continuous.
\end{proof}

\begin{proof}[Proof of Theorem \ref{thm:pseudogradient}]
Fix some $u_0\in X\setminus \text{crit}\, (f)$. Since $\diff f(u_0)\neq 0$, we find $v\in T_{u_0}X$ such that 
$$\diff f(u_0)[v] \leq - \frac 12 \|\mathrm d f(u_0)\|^2,$$
for some constant $\gamma(u_0)>0$. By the continuity of $\diff f$, we can find $r(u_0)>0$ such that 
\begin{equation}
\diff f (u)[v] \leq - \frac 14  \|\mathrm d f(u)\|^2, \quad \forall u \in B_{r(u_0)}(u_0).
\label{eq:vf1}
\end{equation}
Therefore, we set $V_{u_0}(u)\equiv v $.
    
    If $\bar u\in X$ is a critical point of $f$, then by Proposition \ref{prop:hyperbolic} we have that there exist $\mu(\bar u), r(\bar u)>0$ such that 
    \begin{equation}
    \diff f (\bar u+ x) [V_{\bar u} (x)] \leq - \mu(\bar u) \|x\|_{\bar u}^2,\quad \forall x \in X \ \text{with}\ \|x\|<r(\bar u),
    \label{eq:vf2}
    \end{equation}
    with $V_{\bar u}(x) := Lx$, where $L:T_{\bar u} X \to T_{\bar u} X$ is the hyperbolic operator given by the proposition. 
    Without loss of generality we may also assume that the open sets $B_{r(\bar u)}(\bar u), \ \bar u \in \text{crit}\, (f),$ are pairwise disjoint. 
    
    We now consider the open covering of $X$ given by 
    $$\mathfrak U:= \Big \{B_{r(u_0)} (u_0)\ \Big |\ u_0 \in X\setminus \text{crit}\, (f)\Big \} \cup \Big \{B_{r(\bar u)} (\bar u )\ \Big |\ \bar u  \in  \text{crit}\, (f)\Big \}.$$
    By the paracompactness of $X$, there exists a locally finite refinement $\mathfrak V=\{\mathcal V_j \, |\, j \in J\}$ of the open covering $\mathfrak U$. Let $\Gamma:J \to X$ be a function such that 
    $\mathcal V_j \subset B_{r(\Gamma(j))} (\Gamma(j))$ for all $j\in J$. Following \cite{Bonic:1966,Fry:2002}, $X$ admits $C^2$-smooth bump functions. Therefore, we can find a $C^2$-smooth 
    partition of unity $\{\chi_j\}$ subordinated to the open covering $\mathfrak V$, and set 
    $$\tilde F(u) := \sum_{j\in J} \chi_j (u) V_{\Gamma(j)}(u),\quad \forall u \in X.$$
    By construction $\tilde F$ is of class $C^2$, and the inequalities \eqref{eq:vf1} and \eqref{eq:vf2} imply that $f$ is a Lyapounov function for $\tilde F$. Furthermore, we can make $\tilde F$ to a bounded 
    vector field by multiplication by a suitable conformal factor: given a smooth monotonically decreasing function $\varphi:[0,+\infty)\to (0,+\infty)$ such that 
   $$\varphi \equiv 1 \ \ \text{for}\ s \in [0,1], \quad \varphi (s) = \frac 1s \ \ \text{for}\ s \geq 2,$$
    we set 
    $$F(u) := \varphi (\|\tilde F(u)\|) \cdot \tilde F(u), \quad \forall u \in X.$$
    Clearly, $F$ is complete, and $f$ is a Lyapounov function for $F$ as well. Moreover, in a neighborhood of each critical point the vector field $F$ coincides with the linear vector field $x\mapsto Lx$. 
    This implies (iii). We claim now that Palais-Smale sequences for $(f,F)$ are also Palais-Smale sequences for $f$. Indeed, let $(u_n)\subset X$ be a sequence such that 
    $$f(u_n)\to c, \quad \mathrm d f(u_n)[F(u_n)] \to 0.$$
    By Step 1 in the proof of Lemma \ref{lem:ps} we have that $(u_n)$ is a bounded sequence. If $(u_n)$ admits a subsequence converging to a critical point of $f$ then there is nothing to prove. 
    So without loss of generality we can assume that, up to passing to a subsequence if necessary, $(u_n)$ is contained in the complement of a open neighborhood of $\text{crit}\, (f)$. By \eqref{eq:vf1}
    we can find a constant $c>0$ such that 
    $$\|\mathrm d f(u_n)\|^2 \leq - \frac 1c \mathrm d f (u_n)[F(u_n)] = o(1),\quad \text{for}\ n \to +\infty.$$
    This together with Lemma \ref{lem:ps} implies (iv). Finally, since $F$ is of class $C^2$, (v) can be achieved by a suitable generic perturbation, see \cite{Abbondandolo:2006lk} (see also \cite{Asselle:2022}, where the
     case of strongly indefinite functionals is treated). 
\end{proof}


\bibliography{_biblio}
\bibliographystyle{plain}

\end{document}